\documentclass[12pt]{article}

\usepackage{amsmath}
\usepackage{amsthm}
\usepackage{amsfonts}

\newcommand{\mmp}{\mathbb{P}}

\newcommand{\dod}{\overset{d}{\to}}
\newcommand{\tp}{\overset{P}{\to}}

\newcommand{\me}{\mathbb{E}}
\newcommand{\mr}{\mathbb{R}}
\newcommand{\mn}{\mathbb{N}}

\newcommand{\lin}{\underset{n\to\infty}{\lim}}

\newcommand{\lix}{\underset{x\to\infty}{\lim}}
\newcommand{\lit}{\underset{t\to\infty}{\lim}}

\newtheorem{thm}{Theorem}[section]
\newtheorem{lemma}[thm]{Lemma}

\newtheorem{assertion}[thm]{Proposition}
\theoremstyle{definition}

\theoremstyle{remark}
\newtheorem{rem}[thm]{Remark}
\begin{document}
\title{Weak convergence of finite-dimensional distributions of the number of empty boxes in the Bernoulli sieve
}
%
\author{Alexander Iksanov\footnote{Faculty of Cybernetics, Taras
Shevchenko National University of Kyiv, 64/13 Volodymyrska, str.,
Kyiv-01601, Ukraine; e-mail: iksan@univ.kiev.ua}, \ Alexander
Marynych\footnote{Faculty of Cybernetics, Taras Shevchenko
National University of Kyiv, 64/13 Volodymyrska, str., Kyiv-01601,
Ukraine; marynych@unicyb.kiev.ua}, Vladimir
Vatutin\footnote{Steklov Mathematical Institute, 8, Gubkin str.,
119991, Moscow, Russia; e-mail: vatutin@mi.ras.ru}, }
\maketitle
\begin{abstract}
\noindent The Bernoulli sieve is a random allocation scheme
obtained by placing independent points with the uniform $[0,1]$
law into the intervals made up by successive positions of a
multiplicative random walk with factors taking values in the
interval $(0,1)$. Assuming that the number of points is equal to
$n$ we investigate the weak convergence, as $n\to\infty$, of
finite-dimensional distributions of the number of empty intervals
within the occupancy range. A new argument enables us to relax the
constraints imposed in previous papers on the distribution of the
factor of the multiplicative random walk.
\end{abstract}

\section{Introduction and main results}

Let $T:=\big(T_k\big)_{k\in\mn_0}$ be a multiplicative random walk
defined by
$$T_0:=1, \ \ T_k:=\prod_{i=1}^k W_i, \ \ k\in\mn:=\mn_0\backslash
\{0\},$$ where $\mn_0:=\{0,1,2,\ldots\}$,
$\big(W_k\big)_{k\in\mn}$ are independent copies of a random
variable $W$ taking values in the open interval $(0,1)$. Also, let
$\big(U_k\big)_{k\in\mn}$ be independent random variables which
are independent of $T$ and have the uniform $[0,1]$ law. A random
allocation scheme in which 'balls'\, $U_1$, $U_2$ etc. are
allocated over an infinite array of 'boxes'\, $(T_k, T_{k-1}]$,
$k\in\mn$, is called the {\it Bernoulli sieve}. The study of this
allocation scheme was initiated in~\cite{Gne}. Since then
 a number of papers~\cite{slow, GneIksMar, GIM2,
GINR, GIR, Iks, Iks2} has appeared which analyze some asymptotic properties of
the Bernoulli sieve.

Since a particular ball falls into the box $(T_k, T_{k-1}]$ with
a random probability
\begin{equation*} P_k:=T_{k-1}-T_k=W_1W_2\cdots
W_{k-1}(1-W_k),
\end{equation*}
the Bernoulli sieve is also the classical Karlin's allocation
scheme \cite{GnePitHan, Karlin} with the {\it random frequencies}
$(P_k)_{k\in\mn}$ (or in the random environment $\big(P_k\big)$ or
$\big(W_k\big)$). In this setting it is assumed that, given the
environment $\big(P_k\big)$, some {\it abstract} balls are
allocated over an infinite collection of {\it abstract} boxes
$1,2,\ldots$ independently with probability $P_j$ of hitting box
$j$. In the sequel, we say that the box $(T_k, T_{k-1}]$ has index
$k$.

Recall that some infinite random allocation schemes in
\textit{nonrandom environment} were also investigated in
\cite{Mikh, Mira1, Mira2}. It should be emphasized that infinite
allocation schemes radically differ  from the classical allocation
scheme with finitely many positive frequencies (see monograph \cite{Sevast} for more detail).

Assuming that the number of balls to be allocated equals $n$ (in
other words, using a sample of size $n$ from the uniform
distribution on $[0,1]$), denote by $K_n$ the number of occupied
boxes and by $M_n$ the index of the last occupied box. Set
$L_n:=M_n-K_n$ and note that $L_n$ equals the number of empty
boxes with indices not exceeding $M_n$. The articles mentioned in
the first paragraph give a fairly complete account of {\it
one-dimensional convergence} of $K_n$, $M_n$ and $L_n$. The
present paper contains the first results concerning weak
convergence of {\it finite-dimensional} distributions of elements
of the collection $\big(L_n\big)$.

Before formulating the main results of the paper we recall an
assertion given in Theorem 1.1 \cite{Iks2}.
\begin{assertion}\label{old}
If $\me |\log W|=\infty$ and
\begin{equation*}
\lin {\me(1-W)^n\over \me W^n}=c\in (0,\infty),
\end{equation*}
then
\begin{equation}\label{3old}
L_n\ \dod \ L, \ \ n\to\infty,
\end{equation}
where $L$ is a random variable with a geometric law
$$\mmp\{L=k\}={c\over c+1}\bigg({1\over c+1}\bigg)^k, \ \ k\in\mn_0.$$ In particular, relation \eqref{3old} holds if
\begin{equation}\label{5}
\lix \,{\mmp\{|\log W|> x\}\over\mmp\{|\log (1-W)|> x\}}=c.
\end{equation}
\end{assertion}
There are no reasons to expect that the conditions $\me |\log
W|=\infty$ and \eqref{5} alone are sufficient for weak convergence
of some {\it finite-dimensional} distributions related to
$\big(L_n\big)$. Nevertheless,  a result of this sort is given in
Theorem \ref{main} below under an additional assumption imposed on
the {\it decay rate} to zero of the numerator in \eqref{5}.

Let $N^{(\alpha,\,c)}_\infty:=\sum_k \varepsilon_{(t_k,\,j_k)}$ be
a Poisson random measure on $[0,\infty)\times (0,\infty]$ with
mean measure $\mathbb{LEB}\times \nu_{\alpha,\,c}$, where
$\mathbb{LEB}$ is the Lebesgue measure on $[0,\infty)$, and
$\nu_{\alpha,\,c}$ is a measure on $(0,\infty]$ defined by
$$\nu_{\alpha,\,c}\big((x,\infty]\big)=c^{-1}x^{-\alpha}, \ \
x>0.$$ Further, let $\big(X_\alpha(t)\big)_{t\geq 0}$ be an
$\alpha$-stable subordinator which is independent of
$N^{(\alpha,\,c)}_\infty$ and has the Laplace transform
$$\me
\exp(-zX_\alpha(t))=\exp(-\Gamma(1-\alpha)tz^\alpha),\quad z\geq
0,$$ where $\Gamma(\cdot)$ is the gamma function. Denote by
$\big(X_\alpha^\leftarrow(s)\big)_{s\geq 0}$  an inverse
$\alpha$-stable subordinator defined by
$$
X_\alpha^\leftarrow(s):=\inf\{t\geq 0: X_\alpha(t)>s\}, \ \ s\geq
0.
$$

We stipulate hereafter that $\ell$, $\widehat{\ell}$ and
$\ell^\ast$ denote functions  slowly varying at infinity. Besides,
we write $Z_t(u) \overset{{\rm f.d.}}{\Rightarrow} \ Z(u)$,
$t\to\infty$ to denote weak convergence of finite-dimensional
distributions meaning that for any $n\in\mn$ and any selection
$0<u_1<u_2<\ldots<u_n<\infty$
$$\big(Z_t(u_1),\ldots, Z_t(u_n)\big) \ \dod \ \big(Z(u_1),\ldots,
Z(u_n)\big), \ \ t\to\infty.$$
\begin{thm}\label{main}
If there exist $\alpha\in (0,1)$, $c\in (0,\infty)$ and a
function $\ell$ such that
\begin{equation}\label{0030}
\mmp\{|\log W|> x\} \ \ \sim \ \ c\,\mmp\{|\log (1-W)|>x\} \ \sim
\ x^{-\alpha}\ell(x), \ \ x\to\infty,
\end{equation}
then
\begin{equation*}
L_{[e^{ut}]} \  \overset{{\rm f.d.}}{\Rightarrow} \ \sum_k
1_{\{X_\alpha(t_k)\leq u< X_\alpha(t_k)+j_k\}}=:R_{\alpha,\,c}(u),
\ \ t\to\infty.
\end{equation*}
Furthermore, with $u>0$ fixed, the distribution of
$R_{\alpha,\,c}(u)$ is geometric with the success probability
$c(c+1)^{-1}$.
\end{thm}
\begin{rem}
The weak convergence of finite-dimensional distributions stated in
Theorem \ref{main} immediately implies the strict stationarity of the
process $\big(R_{\alpha,\,c}(e^t)\big)_{t\in\mr}$.
\end{rem}

Theorem \ref{main2} and Theorem \ref{main3} given below refine
Theorem 1.1 \cite{Iks} and Theorem 1.2 \cite{Iks2}, respectively,
which  deal with one-dimensional convergence only.
\begin{thm}\label{main2}
Suppose that there exist $0\leq \beta\leq \alpha<1$ $($
$\alpha+\beta>0$ $)$ and functions $\ell$ and $\widehat{\ell}$
such that
\begin{equation}\label{002}
\mmp\{|\log W|>x\} \ \sim \ x^{-\alpha}\ell(x) \ \ \text{and} \ \
\mmp\{|\log (1-W)|>x\} \ \sim \ x^{-\beta}\widehat{\ell}(x), \ \
x\to\infty.
\end{equation}
If $\alpha=\beta$, assume additionally that
$$\lix {\mmp\{|\log W|>x\}\over \mmp\{|\log (1-W)|>x\}}=0$$ and that there exists a nondecreasing function $u(x)$ satisfying
$$\lix {\mmp\{|\log
W|>x\}u(x)\over \mmp\{|\log(1-W)|>x\}} =1.$$ Then
\begin{equation}\label{0010}
{\mmp\{|\log W|> t\}\over \mmp\{|\log (1-W)|> t\}}L_{[e^{ut}]} \
\overset{{\rm f.d.}}{\Rightarrow} \
\int_{[0,\,u]}(u-s)^{-\beta}{\rm
d}X_\alpha^\leftarrow(s)=:W_{\alpha,\,\beta}(u), \ \ t\to\infty.
\end{equation}
\end{thm}
\begin{thm}\label{main3}
Suppose that there exist $\beta\in [0,1)$ and a function
$\widehat{\ell}$ such that
\begin{equation}\label{domain2}
\mmp\{|\log (1-W)|>x\} \ \sim \ x^{-\beta}\widehat{\ell}(x), \ \
x\to\infty.
\end{equation}

\noindent {\rm (a)} If $\sigma^2={\rm Var}\,(\log W)<\infty$ then
\begin{equation}\label{0018}
{L_{[e^{ut}]}-\mu^{-1}\int_0^{ut} \mmp\{|\log (1-W)|>y\}{\rm d}y
\over \sqrt{\mu^{-1}\int_0^t \mmp\{|\log (1-W)|>y\}{\rm d}y}} \
\overset{{\rm f.d.}}{\Rightarrow} \ V(u), \ \ t\to\infty,
\end{equation}
where $\mu:=\me |\log W|<\infty$, and $\big(V(s)\big)_{s\geq 0}$
is a centered Gaussian process with
$$\me V(t)V(s)=t^{1-\beta}-(t-s)^{1-\beta}, \ \ 0\leq s\leq t.$$

\noindent {\rm (b)} Suppose that $\sigma^2=\infty$ and there exists a
function $\ell$ such that
\begin{equation}\label{domain0}
\int_{[0,\,x]} y^2 \mmp\{|\log W|\in {\rm d}y\} \ \sim \ \ell(x),
\ \ x\to\infty.
\end{equation}
Let $c(x)=x^{1/2}\ell^\ast(x)$ be a positive function satisfying
$\lix \,x\ell(c(x))/c^2(x)~=~1.$

\noindent {\rm (b1)} If
\begin{equation}\label{555}
\lix \mmp\{|\log (1-W)|>x\}\big(\ell^\ast(x)\big)^2=0
\end{equation}
then relation \eqref{0018} holds true.

\noindent {\rm (b2)} If $\beta=0$ and, in addition,
$$\lix \mmp\{|\log (1-W)|>x\}\big(\ell^\ast(x)\big)^2=\infty$$ then
$${L_{[e^{ut}]}-\mu^{-1}\int_0^{ut} \mmp\{|\log (1-W)|>y\}{\rm d}y\over \mu^{-3/2}c(t)\mmp\{|\log (1-W)|>t\}} \ \overset{{\rm f.d.}}{\Rightarrow} \
\int_{[0,\,u]}(u-s)^{-\beta}{\rm d}Z_2(s)=:W_{2,\,\beta}(u), \ \
t\to\infty,$$ where $\big(Z_2(s)\big)_{s\geq 0}$ is a Brownian
motion.
\newline {\rm (c)} Suppose that there exist $\alpha\in (1,2)$ and a function $\ell$ such
that
\begin{equation}\label{domain1}
\mmp\{|\log W|>x\} \ \sim \ x^{-\alpha}\ell (x), \ \ x\to \infty.
\end{equation}
Let $c(x)=x^{1/\alpha}\ell^\ast(x)$ be a positive function
satisfying $\lix \,x\ell(c(x))/c^\alpha(x)~=~1.$

\noindent {\rm (c1)} If
\begin{equation}\label{555555}
\lix \mmp\{|\log
(1-W)|>x\}x^{2/\alpha-1}\big(\ell^\ast(x)\big)^2=0
\end{equation}
then relation \eqref{0018} holds true.

\noindent {\rm (c2)} Suppose that $\beta\in [0, 2/\alpha-1]$. If
$\beta=2/\alpha-1$, assume additionally that
\begin{equation}\label{0011}
\lix
\mmp\{|\log(1-W)>x|\}x^{2/\alpha-1}\big(\ell^\ast(x)\big)^2=\infty.
\end{equation}
Then
$${L_{[e^{ut}]}-\mu^{-1}\int_0^{ut} \mmp\{|\log (1-W)|>y\}{\rm d}y\over \mu^{-1-1/\alpha}c(t)\mmp\{|\log (1-W)|>t\}} \ \overset{{\rm f.d.}}{\Rightarrow}
\ \int_{[0,\,u]}(u-s)^{-\beta}{\rm
d}Z_\alpha(s)=:W_{\alpha,\,\beta}(u), \ \ t\to\infty,$$ where
$\big(Z_\alpha(s)\big)_{s\geq 0}$ is an $\alpha$-stable L\'{e}vy
process with
\begin{equation}\label{st1}
\mathbb{E}\exp\{izZ_\alpha(1)\}=\exp\{-|z|^\alpha
\Gamma(1-\alpha)(\cos(\pi\alpha/2)+{\rm i}\sin(\pi\alpha/2)\, {\rm
sign}(z))\}, \ z\in\mr.
\end{equation}
\end{thm}
\begin{rem}\label{ccc}
(I) Existence and the properties of functions $c(t)$ claimed in
parts (b) and (c) of Theorem \ref{main3} are well-known. For
instance,  a function $c(t)$ in part (b) is an asymptotic inverse
to $t^2/\int_{[0,\,t]}y^2{\rm d}\mmp\{|\log W|\in {\rm d}y\}\sim
t^2/\ell(t)$. Consequently, according to Proposition 1.5.15
\cite{BGT}, $c(t) \sim t^{1/2}(L^\#(t))^{1/2}$, where $L^\#(t)$ is
the de Bruijn conjugate of $1/\ell(t^{1/2})$.

\noindent (II) Suppose that the distribution of $|\log W|$ is nonlattice. As
shown in Theorem 1.2 \cite{Iks2}, the weak convergence of
one-dimensional distributions in parts (a), (b1) and (c1) of
Theorem \ref{main3} does not require  the regular variation of $\mmp\{|\log (1-W)|>x\}$. It
suffices to assume that $\me |\log (1-W)|=\infty$ along with all the
other assumptions of the theorem.

\noindent (III) Let $\big(Z_2(s)\big)_{s\geq 0}$ be a Brownian
motion, independent of $\big(V(s)\big)$. A.~Yu. Pilipenko
attracted our attention to the fact that finite-dimensional
distributions of the process $\big(V(s)+Z_2(s^{1-\beta})\big)$
(recall that $\beta\in [0,1)$) coincide with those of a scaled
{\it fractional Brownian motion}. Clearly, if  $\beta=0$ then
finite-dimensional distributions of $\big(V(s)\big)$ coincide with
those of a Brownian motion.
\end{rem}

Two situations, though worth investigating, are ruled out in the
present paper.

\noindent (I) Suppose that the assumptions of part (b) or (c) of
Theorem \ref{main3} hold, and the limits in relation \eqref{555} or
\eqref{555555}, respectively, are finite and nonzero. The authors
do not know whether there is even the one-dimensional convergence
in these cases.

\noindent (II) The theorems just formulated are collected together
as their proofs follow the same approach. Unfortunately, such an
approach is not applicable to multiplicative random walks with
$\me (|\log W|+|\log (1-W)|)<\infty$. For this reason we do not
treat this case here. Under the latter assumption two different
proofs of {\it one-dimensional convergence} of the number of empty
boxes can be found in \cite{GINR, GIR}.

The structure of the paper is as follows. In the context of
problems related to random allocations a Poissonization which is a
transition from the original scheme with deterministic number of
balls to a scheme with random (Poisson) number of balls is a
rather efficient tool. In Section \ref{po} we formulate three
lemmas which are proved in Sections \ref{po1} and \ref{po2},
respectively. While two of these indicate that the Poissonization
of the Bernoulli sieve is expedient, the third takes care of a
de-Poissonization, i.e., a reverse transition. The proofs of
Theorems \ref{main}, \ref{main2} and \ref{main3} are given in
Sections \ref{section_pp}, \ref{proof_main2} and
\ref{proof_main3}, respectively. Finally, some auxiliary results
are collected in the Appendix.

\section{Poissonization and de-Poissonization}\label{po}

Let $\big(\tau_k\big)_{k\in\mn}$ be a Poisson flow with unit
intensity which is independent of the random variables
$\big(U_j\big)$ and the multiplicative random walk $T$. Denote by
$\big(\pi(t)\big)_{t\geq 0}$ the corresponding Poisson process
defined by
$$\pi(t):=\#\{k\in\mn: \tau_k\leq t\}, \ \ t\geq 0.$$

Instead of the scheme with $n$ balls we will work with a {\it
Poissonized} version of the Bernoulli sieve in which, for
$j\in\mn$, the $j$th ball (the point $U_j$) is thrown in the boxes
(the intervals $(T_{k-1}, T_k]$) at the epoch $\tau_j$. Thus, the
random number $\pi(t)$ of balls will be allocated over the boxes
within $[0,t]$. Denote by $\pi_k(t)$ the number of balls which
fall into the $k$th box within $[0,t]$. It is evident that, given
the collection (environment) $\big(T_j\big)$, (1)
the process $\big(\pi_k(t)\big)_{t\geq 0}$ is, for each $k$ a Poisson process
with intensity $P_k=T_{k-1}-T_k$, and (2)  these
processes are {\it independent} for different $k$. It is this latter property which
demonstrates the advantage of the Poissonized scheme over the
original one.

Put $M(t):=M_{\pi(t)}$, $K(t):=K_{\pi(t)}$ and $L(t):=L_{\pi(t)}$.
With this notation in view  $L(t)$ is  the number of empty boxes within
the occupancy range obtained by throwing $\pi(t)$ balls. Recall
that the Bernoulli sieve can be interpreted as the Karlin's
allocation scheme in the random environment $\big(W_k\big)$ which
is given by i.i.d. random variables. The first two auxiliary
results of the present paper reveal that one can investigate the
asymptotics of a relatively simple functional which is determined by
the environment only rather than that of $L(t)$.

Denote by $\big(S_n\big)_{n\in\mn_0}$  a zero-delayed random walk
defined by
$$
S_0:=0, \ \ S_n:=|\log W_1|+\ldots+|\log W_n|, \ \ n\in\mn,
$$
and put $\eta_n:=|\log (1-W_n)|$.
\begin{lemma}\label{red}
If  $\me |\log W|=\infty,$ then
\begin{equation}\label{001infnu}
L(e^{ut})-\sum_{k\geq 0}1_{\{S_k\leq ut<S_k+\eta_{k+1}\}} \
\overset{{\rm f.d.}}{\Rightarrow} \ 0, \ \ t\to\infty.
\end{equation}
\end{lemma}
\begin{lemma}\label{red2}
If $\me |\log W|<\infty,$ then
\begin{equation}\label{001finmu}
{L(e^{ut})-\sum_{k\geq 0}1_{\{S_k\leq ut<S_k+\eta_{k+1}\}}\over
a(t)} \ \overset{{\rm f.d.}}{\Rightarrow} \ 0, \ \ t\to\infty
\end{equation}
for any function $a(t)$ satisfying $\lit a(t)=\infty$. In other
words, finite-dimensional distributions of the process
$\big(L(e^{ut})-\sum_{k\geq 0}1_{\{S_k\leq ut<S_k+\eta_{k+1}\}},\,
t\geq 0\big)$ are tight.
\end{lemma}
Lemma \ref{depois} given below allows us to implement a {\it
de-Poissonization}, i.e., a reverse transition from the scheme
with Poisson number of balls to the original scheme with
deterministic number of balls. Although the papers \cite{GINR,
Iks, Iks2} offer several approaches to the de-Poissonization of
the number of empty boxes, the result of Lemma \ref{depois} is the
strongest one out of those known to the authors, even if only one-dimensional
distributions are considered.
\begin{lemma}\label{depois}
With no assumptions on the expectation of $|\log W|,$
$$L(e^{ut})-L_{[e^{ut}]} \ \overset{{\rm f.d.}}{\Rightarrow} \ 0, \ \ t\to\infty.$$
\end{lemma}

\section{Proof of Lemma \ref{red} and Lemma \ref{red2}}\label{po1}

Put
$$
\nu(t):=\inf\{k\in\mn: S_k>t\}, \ \ t\geq 0,
$$
and denote by $U$ the renewal function generated by the random walk $S_k, k\geq 0$, i.e.,
$$
U(t):=\me \nu(t)=\sum_{k\geq 0}\mmp\{S_k\leq t\}, \ \ t\geq 0.
$$
In the sequel we  repeatedly use  the Blackwell theorem and the
key renewal theorem. Since $\me |\log W|=\infty$, a separate
treatment of the situation when the distribution of $|\log W|$ is
lattice is not needed. Indeed, using the monotonicity of $U$ and
appealing to the Blackwell theorem we conclude that in the lattice
case as well as in the case when the distribution of $|\log W|$ is
nonlattice,
\begin{equation}\label{008}
\lit\big(U(t+h)-U(t)\big)=0,
\end{equation}
for any $h>0$. Thus repeating almost literally the proof of the key
renewal theorem given in  \cite{Resnick}, p.~241 we conclude that
$$
\lit \int_{[0,\,t]}g(t-x){\rm d}U(x)=0 ,$$ provided that  $g$ is a
function directly Riemann integrable on $[0,\infty)$, and
$$
\lit \int_{[t,\,\infty)}g(t-x){\rm d}U(x)=0,
$$ provided that $g$ is directly Riemann integrable on $(-\infty,0]$.

\noindent {\sc Proof of Lemma \ref{red}}. It suffices to prove
that the left-hand side of relation \eqref{001infnu} with $u=1$
converges to zero in probability and to use the Cram\'{e}r-Wold
device. To simplify understanding we divide the proof into several
steps.

\noindent {\sc Step 1}. We intend to show that the maximal index
of boxes discovered by the Poisson process within $[0,e^t]$
satisfies
$$
M(e^t)-\nu(t) \ \tp
\ 0, \ \ t\to\infty.
$$

To this end, put $E(n):=-\log \min(U_1,\ldots, U_n)$ and note that
$M(e^t)=\nu(E(\pi(e^t)))$. As $n\to\infty$, the difference
$E(n)-\log n$ converges in distribution to a random variable
$E^\ast$ obeying the Gumbel distribution. Since the sequence
$\big(E(n)\big)$ is independent of the process $\big(\pi(t)\big)$,
 the difference $E(\pi(e^t))-\log \pi(e^t)$ converges in
distribution, as $t\to\infty$, to $E^\ast$, as well. By the weak law of large numbers
for Poisson processes, $\log \pi(e^t)-t \ \tp \ 0 $, as
$t\to\infty$. Hence
\begin{equation}\label{009}
\lit\big(E(\pi(e^t))-t\big)=E^\ast \ \ \text{in distribution}.
\end{equation}
For brevity, set $R(t):=E(\pi(e^t))$. Using Markov's
inequality and the fact that the renewal function $U(t)$ is
nondecreasing we obtain
\begin{eqnarray*}
\mmp\bigg\{\big(\nu(R(t))-\nu(t)\big)1_{\big\{0<R(t)-t\leq
\gamma\big\}}>\varepsilon\bigg|R(t)\bigg\}&\leq&\varepsilon^{-1}
\me \bigg(\big(\nu(R(t))-\nu(t)\big)1_{\big\{0<R(t)-t\leq
\gamma\big\}}\bigg|R(t)\bigg)\\&=&\varepsilon^{-1}
\big(U\big(t+R(t)-t\big)-U(t)\big)1_{\{0<R(t)-t\leq
\gamma\}}\\&\leq& \varepsilon^{-1} \big(U(t+\gamma)-U(t)\big),
\end{eqnarray*}
for any $\gamma>0$ and $\varepsilon>0.$ This combined with
\eqref{008} yields
$$
\lit \mmp\bigg\{\big(\nu(R(t))-\nu(t)\big)1_{\big\{0<R(t)-t\leq
\gamma\big\}}>\varepsilon\bigg|R(t)\bigg\} =0 \ \ \text{almost
surely}.
$$
Consequently
$$
\big(\nu(R(t))-\nu(t)\big)1_{\big\{0<R(t)-t\leq \gamma\big\}} \
\tp \ 0, \ \ t\to\infty,
$$ by the Lebesgue dominated convergence theorem.

Recalling \eqref{009} and using the absolute continuity of the law of
$E^\ast$ and   the inequality
$$
\mmp\big\{\big(\nu(R(t))-\nu(t)\big)1_{\big\{R(t)-t>\gamma\big\}}>\varepsilon\big\}\leq
\mmp\{R(t)-t>\gamma\},
$$
which holds for any $\gamma>0$ and $\varepsilon>0$, we see that
$$
\underset{t\to\infty}{\overline{\lim}}\,\mmp\big\{\big(\nu(R(t))-\nu(t)\big)1_{\big\{R(t)-t>\gamma\big\}}>\varepsilon\big\}\leq
\mmp\{E^\ast>\gamma\}
$$
and therefore,
$$
\underset{\gamma\to\infty}{\lim}\,\underset{t\to\infty}{\overline{\lim}}\,\mmp\big\{\big(\nu(R(t))-\nu(t)\big)1_{\big\{R(t)-t>\gamma\big\}}>\varepsilon
\big\}=0.
$$
The  estimates above lead to an important relation
$$\big(\nu(R(t))-\nu(t)\big)1_{\big\{R(t)-t>0\big\}} \ \tp \ 0, \ \
t\to\infty.$$ Arguing similarly we arrive at
$$\big(\nu(R(t))-\nu(t)\big)1_{\big\{R(t)-t\leq 0\big\}} \ \tp \ 0, \ \
t\to\infty.$$

\noindent {\sc Step 2}. We are seeking a good approximation for
$K(e^t)$ the number of boxes discovered by the Poisson process
within $[0,e^t]$. More precisely, we  prove that
$$K(e^t)-\sum_{k\geq
0}\bigg(1-\exp\big(-e^{t-S_k}(1-W_{k+1})\big)\bigg) \ \tp \ 0, \ \
t\to\infty.$$

We start with the representation
\begin{equation}\label{bo}
K(e^t)=\sum_{k\geq 1}1_{\{\pi_k(e^t)\geq 1\}},
\end{equation}
where $\pi_k(e^t)$ is the number of balls (in the Poissonized
scheme) landing in the $k$th box within $[0,\,e^t]$. In view of
\begin{equation}\label{0023}
\me \big(K(e^t)|(P_j)\big)=\sum_{k\geq
0}\bigg(1-\exp\big(-e^{t-S_k}(1-W_{k+1})\big)\bigg),
\end{equation}
 to establish the desired approximation it is sufficient
to prove that
\begin{equation}\label{0014}
\lit \me\, {\rm Var}\big(K(e^t)|(P_j)\big)=0.
\end{equation}
Given $\big(P_j\big)$, the indicators in \eqref{bo} are
independent. Hence
\begin{eqnarray*}
\me\, {\rm Var}\big(K(e^t)|(P_j)\big)&=&\me \sum_{k\geq 0}\bigg(\exp\big(-e^{t-S_k}(1-W_{k+1})\big)-\exp\big(-2e^{t-S_k}(1-W_{k+1})\big)\bigg)\nonumber\\
&=&
\int_{[0,\,\infty)}\bigg(\varphi(e^{t-y})-\varphi(2e^{t-y})\bigg){\rm
d}U(y) \label{eq7},
\end{eqnarray*}
where $\varphi(y):=\me e^{-y(1-W)}$. By Lemma \ref{dri} in the
Appendix,  $g_0(y):=\varphi(e^y)-\varphi(2e^y)$ is a
directly Riemann integrable function on $\mr$. Applying now the key renewal
theorem justifies relation \eqref{0014}.

\noindent {\sc Step 3}. We intend to prove the relation
$$Z(t):=\sum_{k\geq
0}\bigg(1-\exp\big(-e^{t-S_k}(1-W_{k+1})\big)\bigg)1_{\{S_k>t\}} \
\tp \ 0, \ \ t\to\infty.$$

According to Lemma \ref{dri},
$g_1(y):=\me\big(1-\exp\big(-e^y(1-W)\big)\big)$ is a directly
Riemann integrable  function on $(-\infty,0]$. Hence
$$\me Z(t)=\int_{[t,\,\infty)}g_1(t-y){\rm d}U(y) \ \to \ 0, \ \
t\to\infty,$$ by the key renewal theorem.

\noindent {\sc Step 4}. We are going to prove the relation
$$Y(t):=\sum_{k\geq
0}\bigg(\exp\big(-e^{t-S_k}(1-W_{k+1})\big)\big)-1_{\{S_k+\eta_{k+1}>t\}}
\bigg)1_{\{S_k\leq t\}} \ \tp \ 0, \ \ t\to\infty.$$

To this end, write $Y(t)$ as the difference of two nonnegative
random functions
\begin{eqnarray*}
Y(t)&=&\sum_{k\geq
0}\exp\big(-e^{t-S_k}(1-W_{k+1})\big)1_{\{S_k+\eta_{k+1}\leq
t\}}\\&-&\sum_{k\geq
0}\bigg(1-\exp\big(-e^{t-S_k}(1-W_{k+1})\big)\bigg)1_{\{S_k\leq
t<S_k+\eta_{k+1}\}}=:Y_1(t)-Y_2(t)
\end{eqnarray*}
and show that $\lit \me Y_i(t)=0$, $i=1,2$. Indeed, according to
Lemma \ref{dri}, the functions
$g_2(y):=\me\exp\big(-e^y(1-W)\big)1_{\{1-W>e^{-y}\}}$ and
$g_3(y):=\me (1-\exp(-e^y(1-W)))1_{\{1-W\leq e^{-y}\}}$ are
directly Riemann integrable on $[0,\infty)$. Hence, by the key
renewal theorem,
$$\me Y_1(t)=\int_{[0,\,t]}g_2(t-y){\rm d}U(y) \ \to \ 0, \ \
t\to\infty$$ and $$\me Y_2(t)=\int_{[0,\,t]}g_3(t-y){\rm d}U(y) \
\to \ 0, \ \ t\to\infty,$$ which is the desired result.

Combining conclusions of the four steps finishes the proof of
Lemma \ref{red}.

\noindent {\sc Proof of Lemma \ref{red2}}. If the
distribution  of $|\log W|$ is nonlattice the proof of Lemma \ref{red2}
proceeds along the same lines as that of Lemma \ref{red}. If the
distribution  of $|\log W|$ is $l$-lattice, for some $l>0$, an
additional argument is only needed for the step 1 of the proof of
Lemma \ref{red}. To implement the steps 2 through 4 one may use
Lemma \ref{dri2} from the Appendix.

\noindent {\sc Step 1}. Fix any $\gamma>0$ and select an $m\in\mn$ such
that $\gamma\leq ml$. With this $\gamma$ and $\varepsilon>0$ in hands, we use
the inequality
\begin{eqnarray*}
\mmp\bigg\{{\nu(R(t))-\nu(t)\over a(t)}1_{\big\{0<R(t)-t\leq
\gamma\big\}}>\varepsilon\bigg|R(t)\bigg\}\leq {U(t+ml)-U(t)\over
\varepsilon a(t)}
\end{eqnarray*}
in combination with the relation $$\lit
\big(U(t+ml)-U(t)\big)={ml\over \me |\log W|}$$ and the Lebesgue
bounded convergence theorem to conclude that
$${\nu(R(t))-\nu(t)\over a(t)}1_{\big\{0<R(t)-t\leq
\gamma\big\}} \ \tp \ 0, \ \ t\to\infty.$$ This completes the
proof for the step 1.

\section{Proof of Lemma \ref{depois}}\label{po2}

It suffices to check that
\begin{equation}\label{005}
K(t)-K_{[t]} \ \tp \ 0 \ \ \text{and} \ \ M(t)-M_{[t]} \ \tp \ 0,
\ \ t\to\infty,
\end{equation}
and to use the Cram\'{e}r-Wold device. In view of the inequality
$\mmp(M(t)\neq M_{[t]})\leq \mmp(K(t)\neq K_{[t]})$, only the
first relation in \eqref{005} needs a proof.

We first show that, for any $x>0$,
\begin{equation}\label{006}
K(t+x\sqrt{t})-K(t-x\sqrt{t}) \ \tp \ 0, \ \ t\to\infty.
\end{equation}
By \eqref{0023}, we have, for large enough $t$,
$$\me \big(K(t+x\sqrt{t})-K(t-x\sqrt{t})\big)=\int_{[0,\,\infty)}\big(\varphi((t-x\sqrt{t})e^{-y})
-\varphi((t+x\sqrt{t})e^{-y})\big){\rm d}U(y),$$ where
$\varphi(y)=\me e^{-y(1-W)}$. As the function $-\varphi^\prime(y)$
is nonincreasing, we infer
$$
\varphi\big((t-x\sqrt{t})e^{-y}\big)-\varphi\big((t+x\sqrt{t})e^{-y}\big)\leq
-\varphi^\prime\big((t-x\sqrt{t})e^{-y}\big)2x \sqrt{t}e^{-y},
$$
by the mean value theorem for differentiable functions, and
therefore
$$
\me \big(K(t+x\sqrt{t})-K(t-x\sqrt{t})\big)\leq {2x\sqrt{t}\over t-x\sqrt{t}}
\int_{[0,\,\infty)}\big(-\varphi^\prime\big((t-x\sqrt{t})e^{-y}\big)\big)(t-x\sqrt{t})e^{-y}{\rm
d}U(y).
$$
By Lemma \ref{dri},
$g_4(y)=-\varphi^\prime(e^y)e^y$ is a directly Riemann integrable function on
$\mr$. This and Lemma \ref{dri2}  yield
$$
\int_{[0,\,\infty)}\big(-\varphi^\prime\big((t-x\sqrt{t})e^{-y}\big)\big)(t-x\sqrt{t})e^{-y}{\rm
d}U(y)=O(1), \ \ t\to\infty.
$$
Hence, for any $x>0$,
$$
\lit \me \big(K(t+x\sqrt{t})-K(t-x\sqrt{t})\big)=0,
$$
which entails \eqref{006}.

The process $\big(K(s)\big)_{s\geq 0}$ is almost surely
nondecreasing. This implies that, for any $x>0$,
\begin{eqnarray*}
|K_{[t]}-K(t)|&=&|K(\tau_{[t]})-K(t)|1_{\{t-x\sqrt{t}\leq
\tau_{[t]}\leq
t+x\sqrt{t}\}}\\&+&|K(\tau_{[t]})-K(t)|1_{\{|\tau_{[t]}-t|>x\sqrt{t}\}}\\&\leq&
K(t+x\sqrt{t})-K(t-x\sqrt{t})+|K(\tau_{[t]})-K(t)|1_{\{|\tau_{[t]}-t|>x\sqrt{t}\}}.
\end{eqnarray*}
Hence, for any $\varepsilon>0$,
\begin{eqnarray*}
\mmp\{|K_{[t]}-K(t)|>2\varepsilon\}&\leq&
\mmp\{K(t+x\sqrt{t})-K(t-x\sqrt{t})>\varepsilon\}\\
&+&\mmp\{|K(\tau_{[t]})-K(t)|1_{\{|\tau_{[t]}-t|>x\sqrt{t}\}}>\varepsilon\}\\&\leq&
\mmp\{K(t+x\sqrt{t})-K(t-x\sqrt{t})>\varepsilon\}+\mmp\{|\tau_{[t]}-t|>x\sqrt{t}\}.
\end{eqnarray*}
Recalling \eqref{006} and using the central limit theorem give
$$
\underset{t\to\infty}{\overline{\lim}}\mmp\{|K_{[t]}-K(t)|>2\varepsilon\}\leq
\mmp\{|\mathcal{N}(0,1)|>x\},
$$
where $\mathcal{N}(0,1)$ stands for a random variable with the
standard normal law. Letting $x\to\infty$ establishes the first
relation in \eqref{005}. The proof is complete.

\section{Proof of Theorem \ref{main}}\label{section_pp}

According to Lemma \ref{depois}, it suffices to show that
\begin{equation*}
L(e^{ut}) \ \overset{{\rm f.d.}}{\Rightarrow} \ R_{\alpha,\,c}(u),
\ \ t\to\infty.
\end{equation*}
Condition \eqref{0030} entails $\me |\log W|=\infty$. Hence
applying Lemma \ref{red} we conclude that the desired assertion is
equivalent to
\begin{equation}\label{0034} \sum_{k\geq
0}1_{\{S_k\leq ut<S_k+\eta_{k+1}\}} \ \overset{{\rm
f.d.}}{\Rightarrow} \ R_{\alpha,\,c}(u), \ \ t\to\infty.
\end{equation}

We  introduce  more notation to be used in this section and
Lemma \ref{weak_joint_conv} in the Appendix:

\noindent $D:=D[0,\infty)$~--~ the Skorohod space of
right-continuous real-valued functions on $[0,\infty)$ which have
finite limits from the left on $(0,\infty)$;

\noindent $M_p ([0,\infty)\times (0,\infty])$~--~ the set of Radon
point measures on $[0,\infty)\times (0,\infty]$ endowed with the
vague topology;

\noindent $C_K([0,\infty)\times (0,\infty])$~--~ the set of
nonnegative continuous functions on $[0,\infty)\times (0,\infty]$
with compact support\label{foot1}\footnote{We alert the reader
that the roles of $0$ and $\infty$ must be interchanged for the
second coordinate so that the sets $[0,a]\times [b,\infty]$ are
compact for $a,b>0$.};

\noindent $\mu_{\alpha,\,c}$~--~a measure on
$(0,\infty]\times(0,\infty]$ such that
$$\mu_{\alpha,\,c}\big\{(u,v):u> x_1 \ \ \text{or} \ \
v>x_2\big\}=x_1^{-\alpha}+c^{-1} x_2^{-\alpha}, \ \ x_1x_2>0,$$
here the constant $c$ is the same as in \eqref{0030}.

Also, recall the notation $N^{(\alpha,\,c)}_\infty$,
$\nu_{\alpha,\,c}$, $\big(X_\alpha(t)\big)$ and $\big(S_n\big)$
introduced in the paragraphs preceding Theorem \ref{main} and
Lemma \ref{red}, respectively.

It is well-known that the condition $\mmp\{|\log W|>x\}\sim
x^{-\alpha}\ell(x)$ with $\alpha\in (0,1)$ ensures
\begin{equation*}
{S_{[ut]}\over c(t)} \ \Rightarrow \ X_\alpha(u), \ \ t\to\infty
\end{equation*}
on $D$ equipped with the Skorohod $J_1$-topology, where $c(t)$ is
any positive function satisfying $\lit tc^{-\alpha}(t)\ell(c(t))
=1$ (such functions do exist, see Remark \ref{ccc}(I)). Also, we
have
\begin{equation}\label{0027} {S_{[ut]-1}\over c(t)} \ \Rightarrow \
X_\alpha(u), \ \ t\to\infty
\end{equation}
under the $J_1$-topology on $D$, where $S_{[ut]-1}=0$ for $0\leq
u<1/t$. Indeed, according to Theorem 3 \cite{Bing}, the obvious
weak convergence of finite-dimensional distributions entails the
weak convergence under the $J_1$-topology, as, for each $t>0$,
paths of the process in the left-hand side of \eqref{0027} are
nondecreasing almost surely, and the limiting subordinator is
stochastically continuous. Further, according to Proposition 3.21
\cite{ResnickBook}, the condition $\mmp\{|\log (1-W)|>x\}\sim c^{-1}
x^{-\alpha}\ell(x)$ entails the convergence
\begin{equation*}
\sum_{k\geq 1}\varepsilon_{(k/t,\, \eta_k/c(t))} \ \Rightarrow \
N^{(\alpha,\,c)}_\infty, \ \ t\to\infty
\end{equation*}
in $M_p([0,\infty)\times (0,\infty])$ equipped with the vague
topology. By the definition of vague convergence the latter is
equivalent to the following one-dimensional convergence
\begin{equation}\label{wcpp}
\sum_{k\geq 1}g\big(k/t,\, \eta_k/c(t)\big) \ \dod \ \sum_m
g(t_m,\,j_m), \ \ t\to\infty
\end{equation}
for any $g\in C_K([0,\infty)\times (0,\infty])$.

If the jumps of random walk had the same law as $|\log W|$ and
were independent of the sequence $\big(\eta_k\big)=\big(|\log
(1-W_k)|\big)$ then relation \eqref{0034} would follow from
Corollary 2.3 \cite{Res}. In the present situation where the
aforementioned independence is absent the proof given by Mikosch
and Resnick still applies except that the relation
\begin{equation}\label{joint convergence}
\bigg({S_{[ut]-1}\over c(t)},\, \sum_{k\geq 1}
\varepsilon_{(k/t,\, \eta_k/c(t))}\bigg) \ \Rightarrow \
(X_\alpha(u), N^{(\alpha,\,c)}_\infty), \ \ t\to\infty
\end{equation}
on $D \times M_p([0,\infty)\times (0,\infty])$ endowed with the
product topology which is generated by $J_1$ and vague topologies
{\it has to be proved}, whereas in the situation treated in
\cite{Res} \eqref{joint convergence} automatically follows.
Indeed, if relation \eqref{joint convergence} holds true then
repeating literally the last fragment of the proof of Theorem 2.1
\cite{Res} enables us to infer
$$\sum_{k\geq 0}\varepsilon_{(S_k/c(t),\,\eta_{k+1}/c(t))} \
\Rightarrow \ \sum_m\varepsilon_{(X_\alpha(t_m),\,j_m)}, \ \
t\to\infty$$ on $M_p([0,\infty)\times (0,\infty])$. A transition
from the last relation to \eqref{0034} may be implemented along
the lines of the proof of Corollary 2.3 \cite{Res}.

According to Lemma \ref{weak_joint_conv}, relation \eqref{joint
convergence} follows if we can prove
$$
\bigg({S_{[ut]-1}\over c(t)}, \sum_{k\geq 1}
g\big(k/t,\,\eta_k/c(t)\big)\bigg) \ \Rightarrow \ \big(X_\alpha
(u), \sum_m g(t_m,\,j_m)\big), \ \ t\to\infty
$$
on $D\times [0,\infty)$, for any function $g\in
C_K([0,\infty)\times (0,\infty])$.

Considering the second coordinates as functions in $D$ which take
constant values and appealing to Lemma 2.6 \cite{Iglehart} we
conclude that it suffices to prove the weak convergence of
finite-dimensional distribution, i.e.,
\begin{equation}\label{fin_dim_joint1}
\sum_{i=1}^n \gamma_i {S_{[u_i t]-1}\over c(t)}+\gamma \sum_{k\geq
1}g\big(k/t, \eta_k/c(t)\big) \ \dod \ \sum_{i=1}^n\gamma_i
X_{\alpha}(u_i)+\gamma \sum_{m}g(t_m,\,j_m), \ \ t\to\infty
\end{equation}
for any $n\in\mn$, any $\gamma, \gamma_1,\ldots, \gamma_n\geq 0$
and any $0< u_1<u_2<\ldots< u_n$, as well as the tightness of the
coordinates.

The tightness of the coordinates follows from \eqref{0027} and
\eqref{wcpp}, respectively. By technical reasons it is more
convenient to check the relation
\begin{equation}\label{fin_dim_joint}
\sum_{i=1}^n \gamma_i {S_{[u_i t]}\over c(t)}+\gamma \sum_{k\geq
1}g\big(k/t, \eta_k/c(t)\big) \ \dod \ \sum_{i=1}^n\gamma_i
X_{\alpha}(u_i)+\gamma \sum_{m}g(t_m,\,j_m), \ \ t\to\infty
\end{equation}
which is equivalent to \eqref{fin_dim_joint1}, by Slutsky's lemma.

Let us check \eqref{fin_dim_joint} by applying the method used to
prove Proposition 3.21 \cite{ResnickBook}. For $z>0$, we have
\begin{eqnarray*}
\phi_t(z)&:=&\me \exp\Big(-z\Big(\sum_{i=1}^n \gamma_i {S_{[u_it]}\over c(t)}+\gamma \sum_{k\geq 1} g\big(k/t,\, \eta_k/c(t)\big)\Big)\Big)\\
&=&\me \exp\Big(-z\Big(\sum_{i=1}^n {\gamma_i\over
c(t)}\sum_{k\geq 1}|\log W_k|1_{\{k\leq u_i t\}}+\gamma\sum_{k\geq
1}
g\big(k/t, \eta_k/c(t)\big)\Big)\Big)\\
&=&\prod_{k\geq 1}\me \exp\Big(-z\Big({|\log W|\over c(t)}\sum_{i=1}^n \gamma_i1_{\{k\leq u_it\}}+\gamma g\big(k/t,\,|\log (1-W)|/c(t)\big)\Big)\Big)\\
&=&\prod_{k\geq 1}\int_{[0,\,\infty)\times
[0,\,\infty)}\big(1-K(z,u,v,k/t)\big)\mmp\Big\{{|\log W|\over
c(t)}\in{\rm d}u,{|\log (1-W)|\over c(t)}\in{\rm d}v\Big\} ,
\end{eqnarray*}
where $K(z,u,v,w):=1-\exp\Big(-z\Big(u\sum_{i=1}^n \gamma_i
1_{\{w\leq u_i\}}+\gamma g(w,v)\Big)\Big)$. Denote by $C$ the
(compact) support of $g$. It is clear that
\begin{eqnarray*}
&&\hspace{-1cm}\int_{[0,\,\infty)\times [0,\,\infty)}K(z,u,v,k/t)\mmp\Big\{{|\log W|\over c(t)}\in{\rm d}u,\,{|\log (1-W)|\over c(t)}\in{\rm d}v\Big\}\\
&\leq & \me \Big(1-\exp\Big(-{z|\log W|\over c(t)}\sum_{i=1}^n
\gamma_i1_{\{k\leq u_i t\}}\Big)\Big)+ \mmp\Big\{\Big({|\log
W|\over c(t)},\,{|\log (1-W)|\over c(t)}\Big)\in C\Big\},
\end{eqnarray*}
for any $k\in\mn$, which implies
\begin{equation}\label{aux1}
\lit \sup_{k\in\mn}\int_{[0,\,\infty)\times
[0,\,\infty)}K(z,u,v,k/t)\mmp\Big\{{|\log W|\over c(t)}\in{\rm
d}u,{|\log (1-W)|\over c(t)}\in{\rm d}v\Big\}=0.
\end{equation}
Obviously, for $x_0>0$ small enough there exists $M=M(x_0)>0$ such
that
$$
0\leq -\log (1-x)-x \leq Mx^2,\qquad 0<x\leq x_0.
$$
This in combination with \eqref{aux1} gives
\begin{eqnarray*}
&&\hspace{-1.5cm}0\leq -\log \phi_t(z)-\sum_{k\geq 1}
\int_{[0,\,\infty)\times [0,\,\infty)} K(z,u,v,
k/t)\mmp\Big\{{|\log W|\over c(t)}\in{\rm d}u,\,
{|\log (1-W)|\over c(t)}\in{\rm d}v\Big\}\\
&\leq& M\sum_{k\geq 1}\Big(\int_{[0,\,\infty)\times [0,\,\infty)}
K(z,u,v, k/t)\mmp\Big\{{|\log W|\over c(t)}\in{\rm d}u,\, {|\log
(1-W)|\over c(t)} \in{\rm d}v\Big\}\Big)^2,
\end{eqnarray*}
for $t$ large enough, and thereupon
\begin{equation}\label{cf_conv}
\lit \Big(-\log \phi_t(z)-\sum_{k\geq 1} \int_{[0,\,\infty)\times
[0,\,\infty)} K(z,u,v, k/t)\mmp\Big\{{|\log W|\over c(t)}\in{\rm
d}u,\, {|\log (1-W)|\over c(t)}\in{\rm d}v\Big\}\Big)=0,
\end{equation}
for each $z>0$, by another appeal to \eqref{aux1}.

Conditions \eqref{0030} entail
\begin{equation}\label{0031}
t\mmp\bigg\{\bigg({|\log W|\over c(t)},\,{|\log (1-W)|\over
c(t)}\bigg)\in \cdot\bigg\} \ \overset{v}{\to} \ \mu_{\alpha,\,c},
\ \ t\to\infty
\end{equation}
and
\begin{equation}\label{00311}
t\mmp\bigg\{{|\log (1-W)|\over c(t)}\in \cdot\bigg\} \
\overset{v}{\to} \ \nu_{\alpha,\,c}, \ \ t\to\infty,
\end{equation}
where $\overset{v}{\to}$ denotes vague convergence of measures.
Observe that $\mu_{\alpha,\,c}$ is a measure concentrated on the
axes. This justifies the independence of the components of the
limiting vector in \eqref{joint convergence} as well as the
integration formula
\begin{equation}\label{integration}
\int_{[0,\,\infty)\times [0,\,\infty)}f(u,v)\mu_{\alpha,\,c}({\rm
d}u\times {\rm d}v)=\alpha\int_0^\infty f(u,0)u^{-\alpha-1}{\rm
d}u+\alpha c^{-1}\int_0^\infty f(0,v)v^{-\alpha-1}{\rm d}v
\end{equation}
which is valid provided the integrals are well-defined.

Relation \eqref{0031} entails
\begin{eqnarray}\label{vag_conv}
\mu^{(t)}({\rm d}x,{\rm d}u,{\rm d}v)&:=&\sum_{k\geq
1}\varepsilon_{k/t}({\rm d}x)\mmp\bigg\{{|\log W|\over c(t)}\in
{\rm
d}u,\,{|\log (1-W)|\over c(t)}\in {\rm d}v\bigg\}\nonumber\\
&\overset{v}{\to}& {\rm d}x \mu_{\alpha,\,c}({\rm d}u\times {\rm
d}v), \ \ t\to\infty,
\end{eqnarray}
where $\varepsilon_{k/t}$ is a probability measure concentrated at
$k/t$, and relation \eqref{00311} implies
\begin{equation}\label{vag_conv2}
\mu^{(t)}({\rm d}x,(0,\infty],{\rm d}v)\ \ \overset{v}{\to} \ \
{\rm d}x \mu_{\alpha,\,c}((0,\infty]\times {\rm d}v), \ \
t\to\infty.
\end{equation}
We use the representation
\begin{eqnarray*}
&&\hspace{-2cm}\sum_{k\geq 1} \int_{[0,\,\infty)\times
[0,\,\infty)} K(z,u,v, k/t)\mmp\bigg\{{|\log W|\over c(t)}\in{\rm
d}u,\, {|\log (1-W)|\over
c(t)}\in{\rm d}v\bigg\}\\
&&\hspace{3cm}=\int_{[0,\,\infty)\times[0,\,\infty)\times
[0,\,\infty)}K(z,u,v,x)\mu^{(t)}({\rm d}x,{\rm d}u,{\rm d}v)
\end{eqnarray*}
to show that
\begin{eqnarray}\label{Kconv}
&&\lit \int_{[0,\,\infty)\times[0,\,\infty)\times
[0,\,\infty)}K(z,u,v,x)\mu^{(t)}({\rm d}x,{\rm d}u,{\rm
d}v)\nonumber\\
&&\qquad\qquad=\int_{[0,\,\infty)\times[0,\,\infty)\times
[0,\,\infty)}K(z,u,v,x){\rm d}x\mu_{\alpha,\,c}({\rm d}u,{\rm
d}v),
\end{eqnarray}
for each $z>0$. The last relation does not follow from
\eqref{vag_conv} automatically, for, with $z$ fixed, the function
$K(z,\cdot)$ is not compactly supported on
$(0,\infty]\times(0,\infty]\times [0,\infty)$. To prove
\eqref{Kconv}, note that in the same space the function
$(u,v,x)\mapsto K(z,u,v,x)1_{\{u\geq 1\text{ or} \ x>u_n\}}$ does
possess a compact support for any fixed $z>0$ (see footnote on
p.~ \pageref{foot1}). Hence it is sufficient to check that
\begin{eqnarray}\label{Kconv2}
&&\lit \int K(z,u,v,x)1_{\{u<1,\, x\leq u_n\}}\mu^{(t)}({\rm
d}x,{\rm d}u,{\rm d}v)\nonumber\\
&&\qquad\qquad= \int K(z,u,v,x)1_{\{u <
1,\,x\leq u_n\}}{\rm d}x\mu_{\alpha,\,c}({\rm d}u,{\rm d}v).
\end{eqnarray}
To simplify notation we only prove this for $n=1$\footnote{The
case of general $n$ can be settled by splitting the domain of
integration over variable $x$ into segments $[u_{i-1},u_i]$ and
checking convergence of integrals over each segment.}. With $z$
fixed, the function
$$
\hat{K}(z,u,v,x):=
\left\{
\begin{array}{ll}
    \frac{1-\exp\big(-z(u\gamma_1+\gamma g(x,v))\big)}{u\gamma_1+\gamma g(x,v)},& u\gamma_1+\gamma g(x,v)\neq 0,\\
    z,& u\gamma_1+\gamma g(x,v)=0
\end{array}\right.
$$
is continuous and bounded which implies that \eqref{Kconv2}
follows if we  prove that
\begin{equation}\label{wconv_meas1}
u1_{\{u<1,\,x\leq u_1\}}\mu^{(t)}({\rm d}x,{\rm d}u,{\rm d}v) \
\Rightarrow \ u1_{\{u<1,\,x\leq u_1\}}{\rm
d}x\mu_{\alpha,\,c}({\rm d}u,{\rm d}v), \ \ t\to\infty,
\end{equation}
and
\begin{equation}\label{wconv_meas2}
g(x,v)1_{\{u<1,\,x\leq u_1\}}\mu^{(t)}({\rm d}x,{\rm d}u,{\rm d}v)
\ \Rightarrow \ g(x,v)1_{\{u<1,\,x\leq u_1\}}{\rm
d}x\mu_{\alpha,\,c}({\rm d}u,{\rm d}v), \ \ t\to\infty.
\end{equation}
Vague convergence in \eqref{wconv_meas1} and \eqref{wconv_meas2}
is a consequence of \eqref{vag_conv}. Indeed, if $f$ is continuous and
compactly supported, the same is true for $u1_{\{u<1,\,x\leq u_1\}}f$ and
$g(x,v)1_{\{u<1,\,x\leq u_1\}}f$. According to Theorem 30.8 (ii)
\cite{Bauer}, the vague convergence can be strengthened to the weak
convergence if we prove that
\begin{equation}\label{full_meas1}
\lit \int u1_{\{u<1,\,x\leq u_1\}}\mu^{(t)}({\rm d}x,{\rm d}u,{\rm
d}v)= \int u1_{\{u<1,\,x\leq u_1\}}{\rm d}x\mu_{\alpha,\,c}({\rm
d}u,{\rm d}v)
\end{equation}
and
\begin{equation}\label{full_meas2}
\lit \int g(x,v)1_{\{u<1,\,x\leq u_1\}}\mu^{(t)}({\rm d}x,{\rm
d}u,{\rm d}v)= \int g(x,v)1_{\{u<1,\,x\leq u_1\}}{\rm
d}x\mu_{\alpha,\,c}({\rm d}u,{\rm d}v).
\end{equation}

The prelimiting expression  in \eqref{full_meas1} can be
written in the form
\begin{eqnarray*}
\sum_{k=1}^{[u_1 t]}\int_{[0,1)\times
(0,\,\infty)}u\mmp\Big\{{|\log W|\over c(t)}\in{\rm d}u,\, {|\log
(1-W)|\over c(t)} \in{\rm d}v\Big\}=\frac{[u_1 t]}{c(t)}\me |\log
W|1_{\{|\log W|< c(t)\}},
\end{eqnarray*}
which, by \eqref{0030}, converges, as $t\to\infty$, to
$\alpha(1-\alpha)^{-1}u_1$, and this is equal to the value of the
right-hand side of \eqref{full_meas1}. We now prove
\eqref{full_meas2}. Set $\hat{g}(x,v):=g(x,v)1_{\{x\leq u_1\}}$
and note that $\hat{g}\in C_K([0,\infty)\times(0,\infty])$. As the
function $\hat{g}(x,v)1_{\{u\geq 1\}}$ has compact support,
equality \eqref{full_meas2} is justified if we  verify that
\begin{equation}\label{full_meas3} \lit \int
\hat{g}(x,v)\mu^{(t)}({\rm d}x,{\rm d}u,{\rm d}v)= \int
\hat{g}(x,v){\rm d}x\mu_{\alpha,\,c}({\rm d}u,{\rm d}v).
\end{equation}
To this end, note that an equivalent form of \eqref{full_meas3} is
\begin{eqnarray*}
\lit \int \hat{g}(x,v)\mu^{(t)}({\rm d}x,(0,\infty],{\rm d}v)=
\int \hat{g}(x,v){\rm d}x\mu_{\alpha,\,c}((0,\infty],{\rm d}v)
\end{eqnarray*}
which holds true by \eqref{vag_conv2}. Thus, relation
\eqref{Kconv} has been proved.

We use integration formula \eqref{integration} and the equality
$g(x,0)=0$ to conclude that
\begin{eqnarray*}
&&\int_{[0,\,\infty)\times[0,\,\infty)\times
[0,\,\infty)}K(z,u,v,x){\rm d}x\mu_{\alpha,\,c}({\rm d}u,{\rm
d}v)\\
&&\qquad= \alpha\int_0^\infty \int_0 ^\infty \Big(1-\exp\Big(-zu
\sum_{i=1}^n \gamma_i1_{\{x\leq u_i\}}\Big)\Big){\rm d}x
u^{-\alpha-1}{\rm d}u\\
&&\qquad\quad+\alpha c^{-1}\int_0^\infty\int_0^\infty
\Big(1-\exp(-z\gamma g(x,v))\Big){\rm d}x v^{-\alpha-1}{\rm
d}v\\&&\qquad\qquad= -\log \me \exp\big(-z \sum_{i=1}^n\gamma_i
X_\alpha(u_i)\big) -\log \me \exp\big(-z\gamma \sum_m
g(t_m,\,j_m)\big)
\end{eqnarray*}
for each $z>0$, which in combination with \eqref{Kconv} and
\eqref{cf_conv} gives \eqref{fin_dim_joint} and thereupon joint
convergence \eqref{joint convergence}.

Left with determining the law of $R_{\alpha,\,c}(u)$, fix
$\delta>0$, put
$$R^{(\delta)}_{\alpha,\,c}(u):=\sum_{k\geq 0}1_{\{S_k\leq u<S_k+\eta_{k+1}, \eta_{k+1}>\delta\}}$$ and use the equality
$$\me e^{-zR^{(\delta)}_{\alpha,\,c}(u)}=\me \exp\bigg(-\int_{[0,\,\infty)}\int_{(\delta,\,\infty]}\big(1-e^{-z1_{\{X_\alpha(s)\leq u<X_\alpha(s)+y\}}}\big)
{\rm d}s\nu_{\alpha,\,c}({\rm d}y)\bigg), \ \ z\geq 0,$$ which is
a particular case of the formula given in   \cite{Res}, p.~136, along
with simple manipulations to obtain
$$\me e^{-zR^{(\delta)}_{\alpha,\,c}(u)}=\me \exp\bigg(-c^{-1}\int_{[0,\,u]}\big((u-s)\vee \delta\big)^{-\alpha}{\rm
d}X^\leftarrow_\alpha(s)(1-e^{-z})\bigg), \ \ z\geq 0.$$ Passing
to the limit $\delta\downarrow 0$ and using the continuity theorem
for Laplace transforms, we see that
$$\me e^{-zR_{\alpha,\,c}(u)}=\me \exp\bigg(-c^{-1}\int_{[0,\,u]}(u-s)^{-\alpha}{\rm
d}X^\leftarrow_\alpha(s)(1-e^{-z})\bigg), \ \ z\geq 0.$$ Thus the
law of $R_{\alpha,\,c}(u)$ is mixed Poisson with (random)
parameter $c^{-1}\int_{[0,\,u]}(u-s)^{-\alpha}{\rm
d}X^\leftarrow_\alpha(s)$. It is shown in \cite{IMM} that the law
of the latter integral is standard exponential. Therefore the law
of $R_{\alpha,\,c}(u)$ is geometric with the success probability
$c(c+1)^{-1}$, as asserted. The proof of Theorem \ref{main} is
complete.

\section{Proof of Theorem \ref{main2}}\label{proof_main2}

By Lemma \ref{depois}, relation \eqref{0010} follows if we
prove that
\begin{equation}\label{003}
{1-F(t)\over 1-G(t)}L(e^{ut}) \ \overset{{\rm f.d.}}{\Rightarrow}
\ W_{\alpha,\,\beta}(u), \ \ t\to\infty,
\end{equation}
where $F(t):=\mmp\{|\log W|\leq t\}$ and
$G(t):=\mmp\{|\log(1-W)|\leq t\}$, $t\geq 0$.

The first condition in \eqref{002} implies $\me |\log W|=\infty$.
Hence, in view of Lemma \ref{red}, it suffices to check that
\begin{equation}\label{004}
{1-F(t)\over 1-G(t)}\sum_{k\geq 0}1_{\{S_k\leq
ut<S_k+\eta_{k+1}\}} \ \overset{{\rm f.d.}}{\Rightarrow} \
W_{\alpha,\,\beta}(u), \ \ t\to\infty.
\end{equation}
By Theorem 2.10 \cite{IMM} we have
$${1-F(t)\over 1-G(t)}\sum_{k\geq 0}\me\big(1_{\{S_k\leq
ut<S_k+\eta_{k+1}\}}\big|S_k\big)={1-F(t)\over 1-G(t)}\sum_{k\geq
0}\big(1-G(ut-S_k)\big)1_{\{S_k\leq ut\}} \ \overset{{\rm
f.d.}}{\Rightarrow} \ W_{\alpha,\,\beta}(u),$$ as $t\to\infty$.
The validity of the equality
$$
\me \bigg(\sum_{k\geq 0}\big(1_{\{S_k\leq t<S_k+\eta_{k+1}\}}-\me
\big(1_{\{S_k\leq t<S_k+\eta_{k+1}\}}\big|S_k\big)
\big)\bigg)^2=\int_{[0,\,t]}G(y)(1-G(y)){\rm d}U(y)
$$
is easily justified. Further, using the lines of proving
Lemma 5.2 \cite{IMM}, one can show that
$$
\int_{[0,\,t]}G(y)(1-G(y)){\rm d}U(y) \ \sim \ {\rm
const}\,{1-G(t)\over 1-F(t)}, \ \ t\to\infty.
$$
This, Markov's inequality and the conditions of the theorem
 imply
$$
{1-F(t)\over 1-G(t)}\sum_{k\geq 0}\big(1_{\{S_k\leq
t<S_k+\eta_{k+1}\}}-\me \big(1_{\{S_k\leq
t<S_k+\eta_{k+1}\}}\big|S_k\big) \big) \ \tp \ 0, \ \
t\to\infty.
$$
Noting that the first multiplier is regularly varying at $\infty$
and using the Cram\'{e}r-Wold device we arrive at
$${1-F(t)\over 1-G(t)}\sum_{k\geq 0}\big(1_{\{S_k\leq
ut<S_k+\eta_{k+1}\}}-\me \big(1_{\{S_k\leq
ut<S_k+\eta_{k+1}\}}\big|S_k\big) \big) \ \overset{{\rm
f.d.}}{\Rightarrow} \ 0, \ \ t\to\infty$$ thereby proving
\eqref{004} and thereupon \eqref{003}. The proof of Theorem
\ref{main2} is complete.

\section{Proof of Theorem \ref{main3}}\label{proof_main3}

Recall the notation $G(t)=\mmp\{|\log (1-W)|\leq t\}$,
$\eta_n=|\log (1-W_n)|$, $n\in\mn$, and that
$\big(S_n\big)_{n\in\mn_0}$ stands for a zero-delayed standard random
walk with jumps $|\log W_k|$. Set also
$$q(t):=\sqrt{\mu^{-1}\int_0^t \big(1-G(y)\big){\rm d}y}.$$
\begin{assertion}\label{mart}
If $\mu=\me |\log W|<\infty$ and condition \eqref{domain2} holds,
then
$$
{\sum_{k\geq 0}1_{\{S_k\leq ut<S_k+\eta_{k+1}\}}-\sum_{k\geq
0}\me\big(1_{\{S_k\leq ut<S_k+\eta_{k+1}\}}\big|S_k\big)\over
q(t)} \ \overset{{\rm f.\,d.}}{\Rightarrow} \ V(u), \ \
t\to\infty.$$
\end{assertion}
\begin{proof}
We only prove weak convergence of two-dimensional
distributions. The general case is unwieldy and does not require
new ideas.

Fix $0<u_1<u_2$. According to the Cram\'{e}r-Wold device, it suffices
to prove that, for any $\gamma_1, \gamma_2\in\mr$,
\begin{equation}\label{0021}
{\sum_{j=1}^2 \gamma_j \sum_{k\geq 0}1_{\{S_k\leq
u_jt\}}\big(1_{\{S_k+\eta_{k+1}>u_jt\}}-(1-G(u_jt-S_k))\big)\over
q(t)} \ \dod \ \gamma_1V(u_1)+\gamma_2V(u_2),
\end{equation}
as $t\to\infty$. Note that
$\gamma_1V(u_1)+\gamma_2V(u_2)$ is  a normally distributed random variable  with zero mean
and variance $ \gamma_1^2 u_1^{1-\beta}+\gamma_2^2
u_2^{1-\beta}+2\gamma_1\gamma_2(u_2^{1-\beta}-(u_2-u_1)^{1-\beta})$.
Introduce the $\sigma$-algebras $\mathcal{F}_0:=\{\Omega,
\oslash\}$, $\mathcal{F}_k:=\sigma\big(W_1,\ldots, W_k\big)$,
$k\in\mn$ and observe that
$$\me \bigg(\sum_{j=1}^2 \gamma_j 1_{\{S_k\leq
u_jt\}}\big(1_{\{S_k+\eta_{k+1}>u_jt\}}-(1-G(u_jt-S_k))\big)\bigg|\mathcal{F}_k
\bigg)=0.$$ Thus, in order to prove \eqref{0021}, one may use a
martingale central limit theorem (Corollary 3.1 \cite{Hall}),
according to which it suffices to verify that
\begin{equation}\label{123}
\sum_{k\geq 0}\me \big(X_{tk}^2|\mathcal{F}_k\big) \ \tp \
\gamma_1^2 u_1^{1-\beta}+\gamma_2^2
u_2^{1-\beta}+2\gamma_1\gamma_2(u_2^{1-\beta}-(u_2-u_1)^{1-\beta}),
\ \ t\to\infty,
\end{equation}
where
$$
X_{tk}:={\sum_{j=1}^2 \gamma_j 1_{\{S_k\leq
u_jt\}}\big(1_{\{S_k+\eta_{k+1}>u_jt\}}-(1-G(u_jt-S_k))\big)\over
q(t)},
$$
and
\begin{equation}\label{124}
\sum_{k\geq 0}\me
\big(X_{tk}^21_{\{|X_{tk}|>\varepsilon\}}|\mathcal{F}_k\big) \ \tp
\ 0, \ \ t\to\infty,
\end{equation}
for all $\varepsilon>0$, hereafter. The inequality $|X_{tk}|\leq
\big(|\gamma_1|+|\gamma_2|\big)/q(t)$ reveals that under our
conditions relation \eqref{124} follows from \eqref{123}.

It can be checked that
\begin{eqnarray*} \sum_{k\geq 0}\me
\big(X_{tk}^2|\mathcal{F}_k\big)&=&{\sum_{j=1}^2\gamma_j^2\sum_{k\geq
0}1_{\{S_k\leq u_jt\}}(1-G(u_jt-S_k))G(u_jt-S_k)\over q^2(t)}\\&+&
{2\gamma_1\gamma_2 \sum_{k\geq 0}1_{\{S_k\leq
u_1t\}}(1-G(u_2t-S_k))G(u_1t-S_k)\over q^2(t)}.
\end{eqnarray*}
Now we  prove that
\begin{equation}\label{0024}
{\sum_{k\geq 0}1_{\{S_k\leq u_1t\}}(1-G(u_2t-S_k))\over
q^2(t)}={\int_{[0,\,u_1]}\big(1-G(t(u_2-y))\big){\rm
d}\nu(ty)\over q^2(t)} \ \to \ u_2^{1-\beta}-(u_2-u_1)^{1-\beta}
\end{equation}
almost surely, as $t\to\infty$. By the strong law of large numbers
for the process $\big(\nu(t)\big)$, the relation
$$\lit {\nu(ty)\over \mu^{-1}t} =y$$ holds almost surely, for all
$y\in [0,u_1]$. In addition,
$$\lit {1-G(t(u_2-y))\over 1-G(t)}=(u_2-y)^{-\beta}$$ uniformly in $y\in [0,u_1]$. Hence\footnote{A similar relation in a more general situation
can be found in Lemma A.6 \cite{Iks2}.}
$$\lit {\int_{[0,\,u_1]}\big(1-G(t(u_2-y))\big){\rm
d}\nu(ty)\over \mu^{-1}t(1-G(t))}=\int_0^{u_1}(u_2-y)^{-\beta}{\rm
d}y=(1-\beta)^{-1}\big(u_2^{1-\beta}-(u_2-u_1)^{1-\beta}\big)$$
almost surely. It remains to note that $\int_0^t
\big(1-G(y)\big){\rm d}y \sim (1-\beta)^{-1}t(1-G(t))$.

The next step is to verify that
\begin{equation}\label{0022}
\lit
{\int_{[0,\,u_1]}\big(1-G(t(u_2-y))\big)\big(1-G(t(u_1-y))\big){\rm
d}\nu(ty)\over q^2(t)}= 0
\end{equation}
almost surely. To this end, fix $\varepsilon\in (0,u_1)$ and use
the monotonicity of $1-G(t)$ to infer
\begin{eqnarray*}
&&
{\int_{[0,\,u_1-\varepsilon]}\big(1-G(t(u_2-y))\big)\big(1-G(t(u_1-y))\big){\rm
d}\nu(ty)\over q^2(t)}\\
&&\qquad\qquad \leq {\big(1-G(\varepsilon
t)\big)\int_{[0,\,u_1-\varepsilon]}\big(1-G(t(u_2-y))\big){\rm
d}\nu(ty)\over q^2(t)}.
\end{eqnarray*}
Relation \eqref{0024}, with $u_1$ replaced by $u_1-\varepsilon$,
allows us to conclude that the right-hand side of the last
inequality tends to zero almost surely, as $t\to\infty$. Further
$${\int_{(u_1-\varepsilon,\,u_1]}\big(1-G(t(u_2-y))\big)\big(1-G(t(u_1-y))\big){\rm
d}\nu(ty)\over q^2(t)}\leq
{\int_{(u_1-\varepsilon,\,u_1]}\big(1-G(t(u_2-y))\big){\rm
d}\nu(ty)\over q^2(t)},$$ and, as $t\to\infty$, the limit of the
right-hand side equals
$(u_2-u_1+\varepsilon)^{1-\beta}-(u_2-u_1)^{1-\beta}$. Sending now
$\varepsilon$ to zero completes the proof of \eqref{0022}. We
have thus proved that $$\lit {\sum_{k\geq 0}1_{\{S_k\leq
u_1t\}}\big(1-G(u_2t-S_k)\big)G(u_1t-S_k)\over
q^2(t)}=u_2^{1-\beta}-(u_2-u_1)^{1-\beta}$$ almost surely.

Let us show that
\begin{equation}\label{0026}
{\sum_{k\geq 0}1_{\{S_k\leq ut\}}\big(1-G(ut-S_k)\big)\over
q^2(t)}={\int_{[0,\,u]}\big(1-G(t(u-y))\big){\rm d}\nu(ty)\over
q^2(t)} \ \tp \ u^{1-\beta},
\end{equation}
as $t\to\infty$. Notice first that, for any $\varepsilon\in
(0,u)$,
$$\lit {\int_{[0,\,u-\varepsilon]}\big(1-G(t(u-y))\big){\rm d}\nu(ty)\over q^2(t)}=u^{1-\beta}-\varepsilon^{1-\beta}$$ almost
surely, which follows along the lines of the proof of
\eqref{0024}. Since the right-hand side of the latter equality
tends to $u^{1-\beta}$, as $\varepsilon\downarrow 0$, the proof of
\eqref{0026} will now be completed by showing that
$$
\underset{\varepsilon\downarrow
0}{\lim}\underset{t\to\infty}{\overline{\lim}}\,\mmp\bigg\{
{\int_{(u-\varepsilon, \, u]}\big(1-G(t(u-y))\big){\rm
d}\nu(ty)\over q^2(t)}>\delta\bigg\}=0,
$$
for any $\delta>0$. By Markov's inequality, the latter relation
holds true, if we can check that
\begin{equation}\label{0025}
\underset{\varepsilon\downarrow
0}{\lim}\underset{t\to\infty}{\overline{\lim}}\,
{\int_{(u-\varepsilon, \, u]}\big(1-G(t(u-y))\big){\rm
d}U(ty)\over q^2(t)}=0.
\end{equation}
Using Lemma \ref{dri3} and then the regular variation of $1-G(t)$
give
$$\int_{((u-\varepsilon)t, \, ut]}\big(1-G(ut-y)\big){\rm d}U(y) \
\sim \ \mu^{-1}\int_0^{\varepsilon t} \big(1-G(y)\big){\rm d}y \
\sim \ \varepsilon^{1-\beta}q^2(t), \ \ t\to\infty,$$ which proves
\eqref{0025}. Therefore, relation \eqref{0026} holds true.

Finally, an argument similar to that used to establish
\eqref{0022} (or, even simpler, analyzing the asymptotics of
expectation) enables us to check that
\begin{equation}\label{0028}
{\sum_{k\geq 0}1_{\{S_k\leq ut\}}\big(1-G(ut-S_k)\big)^2\over
q^2(t)}={\int_{[0,\,u]}\big(1-G(t(u-y))\big)^2{\rm d}\nu(ty)\over
q^2(t)} \ \tp \ 0,
\end{equation}
as $t\to\infty$. Now convergence in probability stated in
\eqref{123} is a consequence of \eqref{0024}, \eqref{0022},
\eqref{0026} and \eqref{0028} (the last two relations should be
used separately for $u=u_1$ and $u=u_2$).
\end{proof}

Now we are ready to prove Theorem \ref{main3}. According to
Proposition \ref{mart}, conditions \eqref{domain2} and $\mu=\me
|\log W|<\infty$ ensure (it is not necessary to assume here that the
distribution of $|\log W|$ belongs to the domain of attraction of
a stable law)
\begin{equation}\label{0012}
{\sum_{k\geq 0}1_{\{S_k\leq ut<S_k+\eta_{k+1}\}}-\sum_{k\geq
0}\me\big(1_{\{S_k\leq ut<S_k+\eta_{k+1}\}}\big|S_k\big)\over
q(t)} \ \overset{{\rm f.d.}}{\Rightarrow} \ V(u), \ \ t\to\infty.
\end{equation}
Assuming further that the assumptions of either of parts (a)
through (c) are in force an application of Theorem 2.7 \cite{IMM}
yields
\begin{eqnarray}\label{0015}
&&{\sum_{k\geq 0}\me\big(1_{\{S_k\leq
ut<S_k+\eta_{k+1}\}}\big|S_k\big)-\mu^{-1}\int_0^{ut}\big(1-G(y)\big){\rm
d}y \over g(t)}\nonumber\\&&\qquad = {\sum_{k\geq
0}\big(1-G(ut-S_k)\big)1_{\{S_k\leq ut\}}-q^2(ut)\over
g(t)}\quad\overset{{\rm f.d.}}{\Rightarrow} \quad
W_{\alpha,\,\beta}(u), \ \ t\to\infty,
\end{eqnarray}
where $\alpha=2$ corresponds to cases (a) and (b), and
$g(t)=\sqrt{\sigma^2\mu^{-3}t}\big(1-G(t)\big)$ in case (a) and
$g(t)=\mu^{-1-1/\alpha}c(t)\big(1-G(t)\big)$ in cases (b) and (c).

\noindent {\sc Cases} (a), (b1) {\sc and} (c1). Our purpose is to
demonstrate that
\begin{equation*}
{L(e^{ut})-q^2(ut) \over q(t)} \ \overset{{\rm f.d.}}{\Rightarrow
} \ V(u), \ \ t\to\infty
\end{equation*}
which is, by Lemma \ref{depois}, sufficient for proving Theorem \ref{main3} in the
cases under consideration.

The assumptions of Theorem \ref{main3} imply $\mu<\infty$ and
$\lit q(t)=\infty$. By Lemma \ref{red2}, the desired convergence
follows if we prove that
\begin{equation*}
{\sum_{k\geq 0}1_{\{S_k\leq ut<S_k+\eta_{k+1}\}}-q^2(ut)\over
q(t)} \ \overset{{\rm f.d.}}{\Rightarrow} \ V(u), \ \ t\to\infty
\end{equation*}
which, in its turn, is a consequence of \eqref{0012} and
\eqref{0015} if we still verify
\begin{equation}\label{0016}
\lit g(t)/q(t)=0.
\end{equation}
To this end, note first that, in view of Proposition 1.5.8
\cite{BGT},
\begin{equation*}
q^2(t) \ \sim \ {\rm const}\,t^{1-\beta}\widehat{\ell}(t), \ \
t\to\infty.
\end{equation*}
In case (a) we have $g^2(t) \sim {\rm
const}\,t^{1-2\beta}(\widehat{\ell}(t))^2$, $t\to\infty$ which
implies \eqref{0016} (note that $\lit \widehat{\ell}(t)=0$ when
$\beta=0$). In case (b1), $g^2(t) \sim {\rm
const}\,t^{1-2\beta}(\ell^\ast(t)\widehat{\ell}(t))^2$,
$t\to\infty$, and the validity of \eqref{0016} is secured by
\eqref{555}. Finally, in case (c1), $g^2(t) \sim {\rm
const}\,t^{2/\alpha-2\beta}(\ell^\ast(t)\widehat{\ell}(t))^2$,
$t\to\infty$, and \eqref{0016} is valid in view of
\eqref{555555}.

\noindent {\sc Cases} (b2) {\sc and} (c2). The previous argument
allows us to conclude that, first, it suffices to prove that
\begin{equation*}
{L(e^{ut})-q^2(ut) \over g(t)} \ \overset{{\rm f.d.}}{\Rightarrow}
\ W_{\alpha,\beta}(u), \ \ t\to\infty,
\end{equation*}
and second, the latter relation is a consequence of the convergence
\begin{equation}\label{0044}
{\sum_{k\geq 0}1_{\{S_k\leq ut<S_k+\eta_{k+1}\}}-q^2(ut)\over
g(t)} \ \overset{{\rm f.d.}}{\Rightarrow} \ W_{\alpha,\,\beta}(u),
\ \ t\to\infty.
\end{equation}

Relation \eqref{0044} follows from \eqref{0012} and \eqref{0015}
if we  show that
\begin{equation}\label{0020}
\lit g(t)/q(t)=\infty.
\end{equation}
We only treat case (c2), since the analysis of case (b2) requires
similar arguments. By the assumptions of the theorem, $g^2(t) \sim {\rm
const}\,
t^{-2\beta+2/\alpha}\big(\ell^\ast(t)\widehat{\ell}(t)\big)^2$,
$t\to\infty$. Consequently, relation \eqref{0020} holds
automatically if $\beta\in [0,\,2/\alpha-1)$ and is secured by
equality \eqref{0011} if $\beta=2/\alpha-1$.  The proof of Theorem
\ref{main3} is complete.

\section{Appendix}

Lemma \ref{dri} is used in the proofs of Lemma \ref{red} and Lemma
\ref{depois}.
\begin{lemma}\label{dri}
Let $\theta$ be a random variable taking values in $(0,1]$. Then
the functions $g_0(y):=\me \exp(-e^y\theta)-\me\exp(-2e^y\theta)$
and $g_4(y):=e^y\me\theta\exp(-e^y \theta)$ are directly Riemann
integrable on $\mr$, the function
$g_1(y):=\me\big(1-\exp\big(-e^y\theta\big)\big)$ is directly
Riemann integrable on the halfline $(-\infty, 0]$, and the
functions
$g_2(y):=\me\exp\big(-e^y\theta\big)1_{\{\theta>e^{-y}\}}$ and
$g_3(y):=\me \big(1-\exp(-e^y\theta)\big)1_{\{\theta\leq
e^{-y}\}}$ are directly Riemann integrable on the halfline
$[0,\infty)$.
\end{lemma}
\begin{proof}
Since the functions $g_i$, $i=0,4$ and $g_3$ are nonnegative it
suffices to check that they are Lebesgue integrable on $\mr$ and
$[0,\infty)$, respectively, and that the functions $e^{-y}g_i(y)$,
$i=0,3,4$ are nonincreasing (see, for instance, the proof of
Corollary 2.17 \cite{Durr}). The first property follows from the
equalities
\begin{eqnarray*}
\int_\mr g_0(y){\rm d}y&=&\int_0^\infty y^{-1}\big(\me
e^{-y\theta}-\me e^{-2y\theta}\big){\rm d}y\\&=&\me \int_0^\infty
y^{-1}\big(e^{-y\theta}-e^{-2y\theta}\big){\rm d}y=\log 2,
\end{eqnarray*}
\begin{equation*}
\int_\mr g_4(y){\rm d}y=\me \int_0^\infty \theta e^{-y\theta}{\rm
d}y=1
\end{equation*}
and the inequality
\begin{eqnarray*}
\int_0^\infty g_3(y){\rm d}y&=&\me \int_1^\infty
y^{-1}\big(1-\exp(-y\theta)\big)1_{\{\theta\leq y^{-1}\}}{\rm
d}y\\&=&\me \int_\theta^1 y^{-1}\big(1-e^{-y}\big){\rm d}y\leq
\int_0^1 y^{-1}\big(1-e^{-y}\big){\rm d}y<\infty,
\end{eqnarray*}
where the last chain of estimates is justified by the change of variable
and condition $\theta\in [0,1]$ a.s. Further, with $z\in (0,1]$
fixed, the function $y^{-1}(1-e^{-yz})e^{-yz}$ is nonincreasing on
$[0,\infty)$. Hence the function $e^{-y}g_0(y)$ is nonincreasing,
too. By the same reasoning, with $z\in (0,1]$ fixed, the functions
$y^{-1}(1-e^{-yz})$ and $1_{\{z\leq y^{-1}\}}$ are nonincreasing
on $(0,\infty)$, hence, so are their product and the function
$e^{-y}g_3(y)$. The monotonicity of $g_4$ is obvious.

Since $g_1$ is nonnegative and nondecreasing, it suffices to show
that it is Lebesgue integrable on $(-\infty,0]$:
$$\int_{-\infty}^0 g_1(y){\rm d}y=\int_0^1 y^{-1}\me\big(1- e^{-y\theta}\big){\rm
d}y\leq \int_0^1\me \theta {\rm d}y \in (0,1].$$

The  function $h(y):=\exp(-e^y)$ is positive and directly Riemann
integrable on $[0,\infty)$. Since $g_2$ is the convolution of $h$
and the distribution function of $|\log \theta|$, it is directly
Riemann integrable on $[0,\infty)$, by Proposition 2.16(d) in
\cite{cinlar}, p.~297.
\end{proof}
Denote by $\big(S_n^\ast\big)_{n\in\mn_0}$ a zero-delayed random
walk with independent increments distributed as a nonnegative
random variable $\xi^\ast$. Set
$$U^\ast(t)=\sum_{n\geq 0}\mmp\{S_n^\ast\leq t\}, \ \ t\in\mr.$$ This notation is used  in the next two assertions.
Although we think Lemma \ref{dri2} may have been known, we give
its complete proof as we have been unable to locate it in the
literature.
\begin{lemma}\label{dri2}
If  $f: \mr \to [0,\infty)$ is a directly Riemann
integrable function on $[0,\infty)$, then
$$\underset{t\to\infty}{\overline{\lim}}\,\int_{[0,\,t]}f(t-y){\rm d}U^\ast(y)<\infty.$$
If $f$ is directly Riemann integrable on $(-\infty,0]$, then
$$\underset{t\to\infty}{\overline{\lim}}\,\int_{[t,\,\infty)}f(t-y){\rm d}U^\ast(y)<\infty.$$
\end{lemma}
\begin{proof}
If the distribution of $\xi^\ast$ is non-lattice, an (even stronger)
assertion follows from the key renewal theorem. Suppose the distribution of
$\xi^\ast$ is $l$-lattice, $l>0$. We only treat the case of direct
Riemann integrability on $[0,\infty)$.

Since $$f(t)\leq \sum_{n\geq 1}\underset{(n-1)l\leq
s<nl}{\sup}\,f(s)1_{[(n-1)l,\, nl)}(t), \ \ t\geq 0,$$ we obtain
\begin{eqnarray*}
\int_{[0,\,t]}f(t-y){\rm d}U^\ast(y)&\leq& \sum_{n\geq
1}\underset{(n-1)l\leq
s<nl}{\sup}\,f(s)\big(U^\ast(t-nl)-U^\ast(t-(n-1)l)\big)\\&\leq&
U^\ast(l)\sum_{n\geq 1}\underset{(n-1)l\leq s< nl}{\sup}\,f(s)
\end{eqnarray*}
having utilized subadditivity of $U^\ast$ on $\mr$ for the last
inequality. It remains to observe that the series on the
right-hand side converges, since $f$ is directly Riemann
integrable.
\end{proof}

Lemma \ref{dri3} is used in the proof of Proposition \ref{mart}.
\begin{lemma}\label{dri3}
Suppose   $f: [0,\infty) \to [0,\infty)$ is a
nonincreasing function, $\lit \int_{[0,\,t]}f(y){\rm d}y=\infty$ and $0<\me
\xi^\ast<\infty$. For $0\leq a<b\leq 1$ the following relation
holds
$$\int_{[at,\,bt]}f(t-y){\rm d}U^\ast(y) \ \sim \ (\me \xi^\ast)^{-1}\int_{(1-b)t}^{(1-a)t}f(y){\rm d}y, \ \
t\to\infty.$$
\end{lemma}
\begin{proof}
If the distribution of $\xi^\ast$ is nonlattice or $1$-lattice, the proof
runs the same path as that of Theorem 4 \cite{Sgib} which
investigates the case $a=0$, $b=1$. If the distribution of $\xi^\ast$ is
$l$-lattice, the distribution of $l^{-1}\xi^\ast$ is $1$-lattice. Hence,
putting $f_l(t):=f(lt)$ we obtain
\begin{eqnarray*}
\int_{[at,\,bt]}f(t-y){\rm
d}U^\ast(y)&=&\int_{[al^{-1}t,\,bl^{-1}t]}f_l(l^{-1}t-y){\rm
d}\sum_{n\geq 0}\mmp\{l^{-1}S_n^\ast\leq y\}\\&\sim& \ {l\over
\me\xi^\ast}\int_{(1-b)l^{-1}t}^{(1-a)l^{-1}t}f_l(y){\rm
d}y={1\over \me\xi^\ast}\int_{(1-b)t}^{(1-a)t} f(y){\rm d}y.
\end{eqnarray*}
\end{proof}

The statement and proof of Lemma \ref{weak_joint_conv} which is
used to demonstrate Theorem \ref{main} retain the notation
introduced in Section \ref{section_pp}.
\begin{lemma}\label{weak_joint_conv}
Let $X_t, t>0$ and $X$ be random elements taking values in $D$,
and $m_t$ and $m$ be random point processes taking values in $M_p
([0,\infty)\times(0,\infty])$. Weak convergence
\begin{equation}\label{jwc1}
\big(X_t,m_t\big) \ \Rightarrow \ \big(X,m\big),\;\;t\to\infty
\end{equation}
under the product topology on $D\times M_p
([0,\infty)\times(0,\infty])$ holds if, and only if, for each
$f\in C_K([0,\infty)\times(0,\infty])$,
\begin{equation}\label{jwc2}
\big(X_t,m_t(f)\big) \ \Rightarrow \ \big(X,
m(f)\big),\;\;t\to\infty
\end{equation}
under the product topology on $D \times [0,\infty)$.
\end{lemma}
\begin{proof}
Suppose \eqref{jwc1} holds. Then, for any fixed function $f\in
C_K([0,\infty)\times(0,\infty])$, the mapping $T_f:D \times
M_p([0,\infty)\times(0,\infty])\to D\times [0,\infty)$ defined by
$T_f(X,m)= \big(X,m(f)\big)$ is continuous in the product
topology, and \eqref{jwc2} follows by the continuous mapping
theorem.

Conversely, suppose \eqref{jwc2} holds. Then $X_t \Rightarrow X$
on $D$, and $m_t(f)\dod m(f)$, as $t\to\infty$. Consequently, the
families $\big(X_t\big)_{t\geq 0}$ and $\big(m_t(f)\big)_{t\geq
0}$ are tight on $D$ and $[0,\infty)$, respectively. Now
Prohorov's theorem ensures that these are relatively compact. By
Lemma 3.20 \cite{ResnickBook}, the family $\big(m_t\big)$ is tight
(hence relatively compact) on $M_p([0,\infty)\times(0,\infty])$.
Then the Cartesian product $\big(X_t, m_s\big)_{t,s\geq 0}$ is
relatively compact on $D \times M_p([0,\infty)\times(0,\infty])$
which implies that the family $\big(X_t,m_t\big)_{t>0}$ is tight
on $D\times M_p([0,\infty)\times(0,\infty])$. It remains to note
that all subsequential limits of the collection $\big(X_t,
m_t\big)_{t>0}$ are equal in distribution, for \eqref{jwc2} holds
for any function $f\in C_K([0,\infty)\times(0,\infty])$.
\end{proof}

\end{document}